\documentclass[11pt,reqno]{amsart}

\allowdisplaybreaks
\usepackage[title]{appendix}
\usepackage{times}
\usepackage{amsmath,amsfonts, amstext,amssymb,amsbsy,amsopn,amsthm}
\usepackage[initials,alphabetic]{amsrefs}
\usepackage{mathrsfs}
\usepackage{bm}
\usepackage{dsfont}
\usepackage{esint}
\usepackage{graphicx}   % for figures
\usepackage{hyperref}
\usepackage[all]{xy}
\usepackage{relsize}
\usepackage{mathtools}
\usepackage{array}
\usepackage{verbatimbox}
\usepackage{xcolor}
\usepackage{dsfont}
\usepackage{graphicx}   % for figures
\usepackage{hyperref}
\usepackage{tikz}
\usepackage{enumitem}

\newcommand{\BR}{\mathbb{R}}
\newcommand{\BC}{\mathbb{C}}
\newcommand{\BZ}{\mathbb{Z}}
\newcommand{\SU}{\operatorname{SU}}
\newcommand{\SO}{\operatorname{SO}}
\newcommand{\Sp}{\operatorname{Sp}}
\newcommand{\Spin}{\operatorname{Spin}}

\newcommand{\seq}{\subseteq}
\newcommand{\wt}{\widetilde}
\newcommand{\non}{\nonumber}
\newcommand{\ad}{\operatorname{ad}}
\newcommand{\Ad}{\operatorname{Ad}}
\newcommand{\p}{\partial}
\newcommand{\vol}{\operatorname{vol}}

\newtheorem{theorem}{Theorem}[section]

\newtheorem{proposition}[theorem]{Proposition}
\newtheorem{lemma}[theorem]{Lemma}

\theoremstyle{definition}
\newtheorem{definition}[theorem]{Definition}
\theoremstyle{remark}
\newtheorem{remark}{Remark}[section]
\theoremstyle{remark}

\theoremstyle{remark}
\newtheorem{example}{Example}[section]

\theoremstyle{remark}

\theoremstyle{remark}

\theoremstyle{remark}

\begin{document}

\title{Partial Regularity of Stable Stationary Harmonic Maps into Certain Lie Groups}
\date{\today}
\author{Jacob Krantz}
\address{Department of Mathematics, Princeton University, Princeton, NJ 08540, USA}
\email{jk9945@princeton.edu}

\maketitle
\begin{abstract}
Let $M$ be a compact Riemannian manifold, and let $G$ be a compact simple Lie group with bi-invariant metric that is not $\Sp(n)$ for $n \geq 8$, $E_{8}$, $F_{4}$, or $G_{2}$. We show that the singular set of any stable stationary harmonic map $u : M \to G$ has Hausdorff codimension at least four. We also find examples of maps into these manifolds with codimension four singularities to show that we cannot reduce the dimension of the singular set any further.
\end{abstract}
\tableofcontents

\section{Introduction}
In 1982, Schoen and Uhlenbeck \cite{SU82} made significant progress in understanding the regularity of minimizing harmonic maps. They showed that the singular set of every minimizing harmonic map between compact manifolds must have Hausdorff codimension at least three. Furthermore, they provided a method to further reduce the dimension of the singular set: showing the nonexistence of nonconstant minimizing cone maps into the same target manifold. With this new tool in hand, Schoen and Uhlenbeck \cite{SU84} were also able to show that when the target manifold is an $n$-sphere, for $n > 2$, we can actually achieve better than codimension three regularity. In particular, for the three-sphere, they gave a proof that every minimizing harmonic map into $S^{3}$ has codimension four singular set by showing there are no nonconstant minimizing cone maps $\BR^{3} \to S^{3}$. They showed that cone stability implies an energy bound on the underlying map $S^{2} \to S^{3}$ which is only satisfied for constant maps. Our proof can be seen as a natural generalization of this argument to other Lie groups with bi-invariant metrics.

After these initial discoveries, it then became a natural question to ask when else we can achieve this improved regularity. As mentioned in the initial 1982 paper \cite{SU82}, a natural choice of target manifold is a Lie group. We will discuss this case below.

Many, including for example Xin \cite{Xin89} and Okayasu \cite{Oka94}, have improved our understanding of the regularity of minimizing harmonic maps in the more general context of homogeneous spaces. They have successfully improved our knowledge of the regularity of minimizing harmonic maps in many cases. More recently, Hong and Wang \cite{HW99} realized that in certain cases stable stationary harmonic maps satisfy similar compactness properties to minimizing harmonic maps. They were able to use this to study the regularity of stable stationary harmonic maps. Lin and Wang \cite{LW06} were able to use their work to study the regularity of stable stationary harmonic maps into spheres. In particular, they were able to confirm that stable stationary harmonic maps into $S^{3}$ have codimension four singular set. Hsu \cite{Hsu05} was able to deduce a similar result to \cite{HW99} when the target manifold has no stable harmonic maps from $S^{2}$. In particular, in this setting he was able to show the analogous compactness theorem. He then used this compactness theorem to show that if the target manifold does not admit any stable harmonic maps from $S^{2}$, then every stable stationary harmonic map into this manifold has codimension three singular set. This shows that stable stationary harmonic maps into compact Lie groups have codimension three singular set because there are no stable $S^{2}$'s in any compact Lie group. It turns out, however, that we can do even better. We will see that in many cases the singular set has codimension at least four. In particular, we show:

\begin{theorem} \label{mainthm}
Let $M$ be a compact manifold, and let $G$ be a compact simple Lie group that is not $\Sp(n)$ for $n \geq 8$, $E_{8}$, $F_{4}$, or $G_{2}$. Then the singular set of every stable stationary harmonic map $u : M \to G$ has Hausdorff codimension at least four.
\end{theorem}

\begin{remark}
Said differently, we will show that all stable stationary harmonic maps from a compact manifold into $\SU(n)$ for any $n$, $\Spin(n)$ for $n \geq 5$, $\Sp(n)$ for $n \leq 7$, $E_{6}$, and $E_{7}$ are smooth away from a subset of the domain of codimension at least four. Note that this completes the proof that stable stationary harmonic maps into $\SO(n)$ are smooth away from a codimension four set because this result was already known for $n \leq 4.$
\end{remark}

\begin{remark}
We will prove this theorem in Section \ref{mainthmsection}. After the proof we will discuss extending these results to certain quotients of these groups.
\end{remark}

Note that in general one should not expect similar regularity for stable stationary harmonic maps into compact symmetric spaces. Indeed, as a consequence of Theorem \ref{stableconeimmersion} below, we see that many compact symmetric spaces admit stable stationary harmonic maps with codimension three singularities. Morally, this is because the Helgason spheres are stable as minimal submanifolds (see \cite{Hel66} and \cite{Ohn87}).

In order to prove Theorem \ref{mainthm}, we must study the harmonic $S^{2}$'s in Lie groups. Since the 1980's, again starting with Uhlenbeck \cite{Uhl89}, many have studied how to write down these harmonic maps from $S^{2}$ to $G$ as simply and explicitly as possible. One key example is the work of Burstall and Rawnsley \cite{BR90}. They provided a factorization theorem which explicitly expresses a given harmonic map as a sequence of so-called ``flag transformations." Their proof shows that, for certain $G$, every harmonic map $S^{2} \to G$ can be constructed explicitly by iterated flag transformations of a constant map. Moreover, they show that each flag transformation increases the energy of the map by an integer multiple of $\pi.$ We will state their result more precisely in Theorem \ref{factorization} below.

To prove Theorem \ref{mainthm}, we show that a factorization of this form is generally incompatible with cone stability. The proof will proceed by contradiction. We assume that we have found a nonconstant map $\phi : S^{2} \to G$ that is cone stable and would like to show this is impossible. We first prove that harmonic maps with stable cone map satisfy an inequality very similar to the usual stability inequality. For the sake of brevity this will be referred to as the cone stability inequality. We then use the factorization of harmonic $S^{2}$ in $G$ of \cite{BR90} to find a variation that will decrease the energy. From here, we vary our map $\phi$ with our given variation and compute the Hessian of the energy at $\phi$ in the direction of this variation. This turns out to be bounded above by a negative constant depending only on the Lie algebra of $G$ and the multiple of the Killing form we choose for our bi-invariant metric. We then evaluate our cone stability inequality to show that the $L^{2}$ norm of our variation is bounded below by negative four times this constant if our map is indeed cone stable. From here, we show using explicit computation that our variation actually fails to satisfy this inequality, contradicting the cone stability inequality. Thus, our harmonic map $S^{2} \to G$ is not cone stable, and consequently, we can apply dimension reduction to conclude partial regularity.
\bigskip
\section{Acknowledgments}
I want to thank Professor Naber for his invaluable guidance and support. I am very grateful for everything that he has taught me and the time he has spent helping me.
\bigskip
\section{Preliminaries and Notation}
Let $(M^{m},g)$ and $(N^{n},h)$ be compact Riemannian manifolds. We will be studying the regularity of maps $u : M \to N$ that are critical points of the Dirichlet energy. Let us assume that $N$ is isometrically embedded in $\BR^{K}$ for some large enough $K.$ We define
\begin{align}
W^{1,2}(M,N) = \{u : M \to \BR^{K} : u \in W^{1,2}(M,\BR^{K}) \mathrm{~and~} u(x) \in N \mathrm{~a.e.}\} 
\end{align}
We define the Dirichlet energy of a map $u \in W^{1,2}(M,N)$ to be
\begin{align}
E[u] = \dfrac{1}{2}\int_{M}|\nabla u|^{2}dx\label{DirichletEnergy}
\end{align}
Now, let $A$ be the second fundamental form of the isometric embedding $N \hookrightarrow \BR^{K}$. The Euler-Lagrange equations for $\eqref{DirichletEnergy}$ can be written explicitly as
\begin{align}
-\Delta u = A(\nabla u, \nabla u) \label{ELeq}
\end{align}
Note that the Euler-Lagrange equations express criticality of $u$ with respect to variations along the image of $u.$ In the case of harmonic maps into Lie groups with bi-invariant metrics, the Euler-Lagrange equations can be written very compactly: 
\begin{align}
d^{*}u^{*}\theta = 0 
\end{align}
where $\theta$ is the left Maurer-Cartan form of our Lie group $G$.
\begin{definition}
A map $u \in W^{1,2}(M,N)$ solving equation \eqref{ELeq} weakly is called a weakly harmonic map.
\end{definition}
When the dimension, $m$, of $M$ is at most two, every weakly harmonic map is known to be smooth. However, once $m$ is at least three, singularities can form. To this end, we define the regular set
\begin{align}
\operatorname{Reg}(u) = \{x \in M : u \mathrm{~is~smooth~in~some~ball~around~}x\}
\end{align}
and the singular set
\begin{align}
\operatorname{Sing}(u) = M \setminus \operatorname{Reg}(u).
\end{align}
Weakly harmonic maps can have a very large singular set. In fact, Rivi\`ere \cite{Riv95} has constructed examples of harmonic maps that are everywhere discontinuous. This complete lack of regularity is a result of the fact that when $m$ is at least three, solving the Euler-Lagrange equations is insufficient to be a critical point of the energy functional. In particular, there can exist so-called ``domain variations" for which a weak solution is not a critical point. Domain variations are constructed by pulling back a map $u$ by the flow generated by a smooth compactly supported vector field on $M$. Requiring that $u$ must also be critical with respect to these domain variations, we obtain the stationary equation:
\begin{align}
\operatorname{div}\left (\dfrac{1}{2}|\nabla u|^{2}g - u^{*}h\right ) = 0. \label{stateq}
\end{align}
Note that of course this is a weak equality.
\begin{definition}
A map $u$ that solves \eqref{ELeq} and \eqref{stateq} will be called a stationary harmonic map. Such maps are critical points of the energy \eqref{DirichletEnergy}. 
\end{definition}
By an argument of Bethuel \cite{Bet93}, stationary harmonic maps are known to be smooth away from a Hausdorff codimension two measure zero set. If we additionally assume control over the second variation, it is sometimes possible to attain much better regularity. To this end, the second variation of the energy is given by the following formula:
\begin{align}
I^{u}(X) = \nabla^{2}E[u](X,X) = \int_{M}|\nabla^{u} X|^{2} - h(R^{N}(X,\nabla_{i} u)\nabla_{i} u, X)dx \label{stabineq}
\end{align}
where $X \in \Gamma(M,u^{*}TN)$ is compactly supported. Here, we write $\nabla^{N}$ for the Levi-Civita connection of $(N,h)$ and $\nabla^{u}$ for the pullback connection $u^{*}\nabla^{N}$.
\begin{definition}
We say that a weakly harmonic map is stable if for every compactly supported $X \in \Gamma(M,u^{*}TN)$, we get $\nabla^{2}E[u](X,X) \geq 0.$ A stationary map satisfying this stability inequality will be called a stable stationary harmonic map.
\end{definition}
\begin{definition}
A map that minimizes the energy \eqref{DirichletEnergy} over all maps with the same fixed boundary data will be called a minimizing harmonic map.
\end{definition}
Schoen and Uhlenbeck show in \cite{SU82} that the nonexistence of nonconstant minimizing cone maps implies a lack of low-codimension singularities for minimizing maps:
\begin{theorem}[\cite{SU82}]
If there are no nonconstant minimizing cone maps $\BR^{k} \to N$ for $k = 3,\ldots,\ell$, then for any minimizing harmonic map $u : M \to N$, the Hausdorff dimension of $\operatorname{Sing}(u)$ is at most $m - \ell - 1.$
\end{theorem}
Essentially, this is because the cone maps are the potential singularities. Here, we require only that our harmonic map be both stationary and stable. In this case, we can apply the relevant compactness theory in \cite{Hsu05}. In this paper, Hsu was able to show that the analogous result holds for stable stationary harmonic maps when there are no stable harmonic maps $S^{2} \to N$. Such is the case when $N$ is a compact Lie group.
\bigskip
\section{Harmonic Maps into Lie Groups via Burstall-Rawnsley}
Let $G$ be a compact simple Lie group. In this section, we will outline the procedure for factoring weakly harmonic maps $\phi : S^{2} \to G$ given by Burstall and Rawnsley in \cite{BR90}. In order to describe this factorization, we will first need to review the basics of root systems and parabolic subalgebras. For more details, the reader is encouraged to read \cite{BR90}. Our presentation of this material will follow theirs very closely.

We will start off with root systems. Again, let $G$ be a compact simple Lie group. Let $\mathfrak{g}$ be the Lie algebra of $G$, and call its complexification $\mathfrak{g}^{\BC}$. Let $\mathfrak{t}$ be a maximal toral subalgebra of $\mathfrak{g}$ with complexification $\mathfrak{t}^{\BC} \seq \mathfrak{g}^{\BC}$. We can see that $\mathfrak{t}^{\BC}$ is a Cartan subalgebra of $\mathfrak{g}^{\BC}$, i.e., it is a maximal abelian subalgebra consisting of semisimple elements. 
\begin{definition}
The roots of $\mathfrak{g}^{\BC}$ with respect to our choice of Cartan subalgebra are the nonzero elements $\alpha \in (\mathfrak{t}^{\BC})^{*}$ such that the space 
\begin{align}
\mathfrak{g}^{\alpha} = \{X \in \mathfrak{g}^{\BC}: \mathrm{for~all~}H \in \mathfrak{t}^{\BC}, \ad(H)X = \alpha(H)X\}
\end{align}
is not equal to $\{0\}$. The set of roots will be denoted $\Delta = \Delta(\mathfrak{g}^{\BC},\mathfrak{t}^{\BC})$, and $\mathfrak{g}^{\alpha}$ will be called a root space when it is not equal to $\{0\}$.
\end{definition}

We can split the set of roots in half in a natural way to make a choice of the so-called positive roots $\Delta^{+} \seq \Delta$.

\begin{definition}
Let $\Delta^{+} \seq \Delta$. We say that $\Delta^{+}$ is a choice of a positive root system if all of the following hold:
\begin{itemize}
\item Let $\alpha,\beta \in \Delta^{+}$. If $\alpha + \beta$ is a root, then $\alpha + \beta \in \Delta^{+}.$
\item $\Delta^{+} \cup -\Delta^{+} = \Delta$
\item $\Delta^{+} \cap -\Delta^{+} = \varnothing.$
\end{itemize}
\end{definition}

These positive roots contain a further subset called the simple roots. The simple roots are the positive roots that cannot be written as the sum of any two positive roots. Call the simple roots $\alpha_{1},\ldots,\alpha_{\ell}$. Here, $\ell = \dim \mathfrak{t}.$ Any root can be decomposed into an integer linear combination of the simple roots, and any positive root will have non-negative coefficients in the $\alpha_{i}.$

The roots are useful for us because they give us a clean way to construct parabolic subalgebras of $\mathfrak{g}^{\BC}$. In particular, there is a bijection between conjugacy classes of parabolic subalgebras and subsets of the simple roots. To begin with, we should define a parabolic subalgebra:
\begin{definition}
A parabolic subalgebra $\mathfrak{q} \seq \mathfrak{g}^{\BC}$ is one that contains a maximal solvable subalgebra.
\end{definition}
We can now talk about how to construct all of the parabolic subalgebras of $\mathfrak{g}^{\BC}$. 
\begin{theorem}[\cite{BR90} Theorem 4.1, \cite{FH91} pp.~395]\label{parabolictheorem}
Let $I \seq \{1,\ldots,\ell\}$ be the index set of a subset of the simple roots $\alpha_{1},\ldots,\alpha_{\ell}$ corresponding to a fixed choice of positive root system $\Delta^{+} \seq \Delta(\mathfrak{g}^{\BC},\mathfrak{t}^{\BC})$. Now, given a root $\alpha$, we can express it in terms of the simple roots: $\alpha = \sum_{i = 1}^{\ell}n_{i}\alpha_{i}$. Define $n_{I}(\alpha) = \sum_{i \in I}n_{i}.$ The subalgebra 
\begin{align}
\mathfrak{q}_{I} = \mathfrak{t}^{\BC} \oplus \sum_{n_{I}(\alpha) \geq 0}\mathfrak{g}^{\alpha}
\end{align}
is parabolic. Every parabolic subalgebra is conjugate to $\mathfrak{q}_{I}$ for some choice of $I$.
\end{theorem}
It turns out that we can determine the entire parabolic by a single element of $\mathfrak{g}.$ Let $\mathfrak{q}$ be a parabolic subalgebra of $\mathfrak{g}^{\BC}$. Find a maximal toral subalgebra $\mathfrak{t} \seq \mathfrak{q} \cap \mathfrak{g}$. Complexify $\mathfrak{t}$ to get a Cartan subalgebra $\mathfrak{t}^{\BC}$ of $\mathfrak{g}^{\BC}.$ Refining Theorem \ref{parabolictheorem}, it turns out that we can find a choice of $\Delta^{+} \seq \Delta(\mathfrak{g}^{\BC},\mathfrak{t}^{\BC})$ such that $\mathfrak{q} = \mathfrak{q}_{I}$ for an index set corresponding to a subset of the simple roots. 

By skew-symmetry of $\ad$ with respect to the Killing form, we conclude that $\Delta(\mathfrak{g}^{\BC},\mathfrak{t}^{\BC}) \seq \sqrt{-1}\mathfrak{t}^{*}.$ Now, if again $\alpha_{1},\ldots,\alpha_{\ell}$ are our simple roots as before, we define dual vectors $\xi_{1},\ldots,\xi_{\ell} \in \mathfrak{t}$ by $\alpha_{i}(\xi_{j}) = \sqrt{-1}\delta_{ij}$. This is possible because $\ell = \dim \mathfrak{t}.$ We set $\xi = \sum_{i \in I}\xi_{i}$ and call this $\xi$ the canonical element of $\mathfrak{q}$. 
\begin{definition}
The element $\xi = \sum_{i \in I}\xi_{i}$ of $\mathfrak{t}$ is called the canonical element of $\mathfrak{q}.$
\end{definition}
It turns out that the eigenvalues of $\ad(\xi)$ are elements of $\sqrt{-1}\BZ$. If we define $\mathfrak{g}_{k}$ to be the $\sqrt{-1}k$-eigenspace of $\ad(\xi)$, we get a nice decomposition $\mathfrak{q} = \sum_{k \geq 0}\mathfrak{g}_{k}$. We also get a decomposition of the nilradical of $\mathfrak{q}$, $\mathfrak{n} = \sum_{k \geq 1}\mathfrak{g}_{k}.$ It will also be useful to know that the Adjoint orbit of a given canonical element is a flag manifold embedded Ad-equivariantly in $\mathfrak{g}$. This means that a flag manifold can be viewed as a set of canonical elements. Furthermore, every flag manifold of $G$ can be understood in this way.

These canonical elements are important for us because they describe a natural way to build harmonic maps. The first step in this direction is defining flag factors. Let $F$ be a flag manifold of $G$, embedded in $\mathfrak{g}$ as described above. It is useful to note that if $G^{\BC}$ is the complexification of $G$, we can find a parabolic subgroup $P$ of $G^{\BC}$ such that 
\begin{align}
F = G^{\BC}/P = G/(G \cap P) \hookrightarrow \mathfrak{g}.
\end{align}
Next, let $\sigma : S^{2} \to F \seq \mathfrak{g}$ be a smooth map. By construction, for $x \in S^{2}$, $\sigma(x)$ is a canonical element of a parabolic subalgebra of $\mathfrak{g}^{\BC}$. Therefore, $\sigma$ defines a bundle of parabolic subalgebras over $S^{2}$: $\mathfrak{q}^{\sigma} \seq \underline{\mathfrak{g}}^{\BC} = S^{2} \times \mathfrak{g}^{\BC}$. The fiber over $x \in S^{2}$ is the unique parabolic subalgebra of $\mathfrak{g}^{\BC}$ with canonical element $\sigma(x).$ Denote this parabolic subalgebra by $\mathfrak{q}^{\sigma}|_{x}.$ Define $\mathfrak{g}_{k}|_{x} \seq \mathfrak{q}|_{x}$ to be the $\sqrt{-1}k$-eigenspace of $\ad(\sigma(x)).$ We then let $\mathfrak{g}_{k}$ be the full $\sqrt{-1}k$-eigenbundle of $\ad(\sigma)$. 

As before, we define $\theta$ to be the left Maurer-Cartan form of $G$. We also define $\nabla^{\phi}$ to be the pullback of the Levi-Civita connection of $G$ by $\phi$. More precisely, we will define it directly as a connection on the trivial bundle $\underline{\mathfrak{g}}$ by $\nabla^{\phi} = d + \frac{1}{2}\ad(\phi^{*}\theta).$ We then extend this complex linearly to a connection on the complexification $\underline{\mathfrak{g}}^{\BC}$.

\begin{definition}
Let $\phi : S^{2} \to G$ be a smooth map. The map $\sigma$ above is called a flag factor for $\phi$ if both
\begin{itemize}
\item $\phi^{*}\theta^{1,0} \in \mathfrak{g}_{0} \oplus \mathfrak{g}_{1}$
\item $(\nabla^{\phi}\sigma)^{0,1} \in \mathfrak{g}_{1}$
\end{itemize}
Recall that the $\mathfrak{g}_{k}$ are the $\sqrt{-1}k$-eigenbundles of $\ad(\sigma)$, so the first condition does indeed depend on $\sigma$. We call the map $\phi\exp(\pi\sigma)$ the flag transformation of $\phi$ by $\sigma.$ 
\end{definition}
In Theorem \ref{factorization} we will see that flag factors can actually be constructed and that they help us factor our harmonic map into simpler pieces. When the flag manifold is actually a Hermitian symmetric space, one should think of the second condition above as an equivalent condition for the holomorphy of $\mathfrak{q}^{\sigma}$ (see \cite{BR90} page 92). 

Note that, for example, the notation $\phi^{*}\theta^{1,0}$ denotes extending $\phi^{*}\theta$ complex linearly and then computing the $(1,0)$-part. It is useful to view these real maps as complex because we have a nice decomposition of $\underline{\mathfrak{g}}^{\BC}$ in terms of the $\mathfrak{g}_{k}.$ Namely, $\underline{\mathfrak{g}}^{\BC} = \sum_{k}\mathfrak{g}_{k}.$

Note also that flag transformations preserve harmonicity. Furthermore, a harmonic map and its flag transformation differ in energy by an integer multiple of $16\pi n_{G}$, where $n_{G}$ is the reciprocal of the norm squared of the highest root in $G.$ Table \ref{table:1}, from \cite{BR90}, shows the values of $n_{G}$ for the Killing form. Note that $n_{G}$ is not to be confused with a standalone $n$. Any use of $n_{G}$ will always be accompanied by the appropriate subscript.
\begin{table}
\centering
\addvbuffer[0pt 10pt]{\begin{tabular}{|W{c}{2.5cm}|W{c}{1cm}|}
\hline
$G$ & $n_{G}$ \\ 
\hline
$\SU(n)$ & $n$ \\
$\Spin(n)$, $n \geq 5$ & $n - 2$ \\
$\Sp(n)$ & $n + 1$ \\
$E_{6}$ & $12$ \\
$E_{7}$ & $18$ \\
$E_{8}$ & $30$ \\
$F_{4}$ & $9$ \\
$G_{2}$ & $4$ \\
\hline
\end{tabular}}
\caption{Values of $n_{G}$ for the compact simple Lie groups}
\label{table:1}
\end{table}
Now, we make our main restriction on $G$:
\begin{definition}
We say that a compact simple Lie group is of type H if it admits a Hermitian symmetric quotient. This includes all of the classical compact simple Lie groups and just rules out $E_{8}$, $F_{4}$, and $G_{2}.$
\end{definition}
Unless otherwise stated, we will assume $G$ is of type H. Additionally, we will eventually make the assumption that $G$ is not $\Sp(n)$ for $n > 7.$ This is for technical reasons that we will see later. It turns out that when our group is of type H, any harmonic $S^{2}$ admits a factorization by iterated flag transformations. More precisely,
\begin{theorem}[\cite{BR90} Theorem 8.9]\label{factorization}
Let $\phi : S^{2} \to G$ be a nonconstant harmonic map. Let $H$ be Hermitian symmetric quotient of $G.$ There exists a harmonic map $\phi_{0} : S^{2} \to G$ and a flag factor $\sigma : S^{2} \to H \seq \mathfrak{g}$ for $\phi_{0}$ such that $\phi = \phi_{0}\exp(\pi\sigma).$ Additionally, the difference of the energy of $\phi$ and $\phi_{0}$ is bounded below:
\begin{align}
E[\phi] - E[\phi_{0}] \geq 16\pi n_{G}.
\end{align}
Continuing this process, which must terminate eventually because $E[\phi] < \infty$, one may completely factor $\phi$ in terms of flag transformations.
\end{theorem}
In other words, $\phi$ can be constructed as a finite number of flag transformations of a constant map, and furthermore, with each flag transformation, the energy must go up by a positive integer multiple of $16\pi n_{G}.$ From this theorem, it seems quite possible that $\sigma$ will contribute to the index of the cone map of $\phi$. We will soon see that this is indeed the case.
\bigskip
\section{Cone Stability Inequality}
In this section, we reduce the full stability inequality for a cone map to a simple inequality on the underlying sphere map. Let us first be more precise about our notation.
\begin{definition}
Let $N$ be a compact manifold, and let $\phi : S^{n} \to N$ be a smooth map where $n \geq 2$. The cone map associated to $\phi$ is the map $\wt\phi : \BR^{n + 1} \to N : x \mapsto \phi(\frac{x}{|x|}).$
\end{definition}
\begin{definition}
We say that a smooth map $\phi : S^{n} \to N$ is cone stable if its cone map is a stable harmonic map.
\end{definition}
\begin{example}
It is well-known that harmonicity of $\phi$ is equivalent to harmonicity of $\wt\phi$ in the weak sense. However, we shall see below that if $\phi$ is unstable, it is not at all obvious whether it is cone stable. In fact, the cone stability inequality \eqref{conestabineq} below is the usual stability inequality for maps from the sphere with an additional positive term. Therefore, if we plug in an unstable vector field for the sphere map, \eqref{conestabineq} may still be satisfied.

A simple example of a harmonic map that is unstable but cone stable is the identity map $\phi : S^{n} \to S^{n}$ for $n \geq 3.$ Indeed, $\phi$ is an unstable map \cite{Smi75}, but its cone map is in fact minimizing \cite{Lin87}.
\end{example}
\begin{proposition}
Suppose $S^{n}$ is equipped with the round metric of radius one. Let $(N,h)$ be any compact manifold, and assume that $\phi : S^{n} \to N$ is cone stable with associated cone map $\wt\phi.$ Then, 
\begin{align}
\int_{S^{n}}\dfrac{(n - 1)^{2}}{4}|X|^{2} + |\nabla^{\phi}X|^{2} - h(R^{N}(X,\nabla_{i}\phi)\nabla_{i}\phi,X) \operatorname{vol}_{S^{n}}\geq 0 \label{conestabineq}
\end{align}
for all $X \in \Gamma(S^{n}, \phi^{*}TN)$.
\end{proposition}
\begin{proof}
To begin, let us recall the original stability inequality. We know that $\wt\phi$ is stable, so
\begin{align}
I^{\wt\phi}(\wt X) = \int_{\BR^{n + 1}}|\nabla^{\wt\phi}\wt X|^{2} - h(R^{N}(\wt X, \nabla_{i}\wt\phi)\nabla_{i}\wt\phi, \wt X)dx \geq 0
\end{align}
for every compactly supported $\wt X \in \Gamma(\BR^{n + 1}, \wt\phi^{*}TN)$. Now, we choose a nice section $\wt X = \psi(r)X(\omega)$, where $\psi \in C_{c}^{\infty}((0,\infty);\BR)$ and $X \in \Gamma(S^{n},\phi^{*}TN).$ This is certainly a section of $\wt\phi^{*}TN$ because at $r\omega \in \BR^{n + 1}$, 
\begin{align}
\wt X(r\omega) = \psi(r)X(\omega) \in T_{\phi(\omega)}N = T_{\wt\phi(r\omega)}N 
\end{align}
and $\psi$ is supported away from zero. Now, choose the local orthonormal frame near $x \in \BR^{n + 1}$ given by the radial vector field $\frac{\p}{\p r}$ and the vector fields $\{e_{j}\}_{j = 1}^{n}$ along the sphere. We will use such a frame for all of our computations below. We may now plug into the stability inequality:
\begin{align}
I^{\wt\phi}(\wt X) &= \int_{\BR^{n + 1}}|\nabla^{\wt\phi}\wt X|^{2} - h(R^{N}(\wt X, \nabla_{\frac{\p}{\p r}}\wt\phi)\nabla_{\frac{\p}{\p r}}\wt\phi, \wt X) - h(R^{N}(\wt X, \nabla_{i}\wt\phi)\nabla_{i}\wt\phi, \wt X)dx \\ \non \\
&= \int_{\BR^{n + 1}}|\nabla^{\wt\phi}\wt X|^{2} - h(R^{N}(\wt X, \nabla_{i}\wt\phi)\nabla_{i}\wt\phi, \wt X)dx \\ \non \\
&= \int_{0}^{\infty}\int_{S^{n}}|\nabla^{\wt \phi}(\psi(r)X(\omega))|^{2}r^{n} - h(R^{N}(\psi(r) X(\omega), \nabla_{i}\wt\phi)\nabla_{i}\wt\phi, \psi(r)X(\omega))r^{n}\vol_{S^{n}}(\omega)dr \\ \non \\
&= \int_{0}^{\infty}\int_{S^{n}}|\nabla^{\wt\phi}_{\frac{\p}{\p r}}(\psi(r)X(\omega))|^{2}r^{n} + |\nabla^{\wt\phi}_{j}(\psi(r)X(\omega))|^{2}r^{n} \\ \non
&\qquad \qquad\qquad- \psi(r)^{2}h(R^{N}(X(\omega),\nabla_{i}\wt\phi)\nabla_{i}\wt\phi, X(\omega))r^{n}\vol_{S^{n}}(\omega)dr \\ \non \\
&= \int_{0}^{\infty}\int_{S^{n}}\psi'(r)^{2}|X(\omega)|^{2}r^{n} + \psi(r)^{2}|\nabla^{\wt\phi}_{j}X(\omega)|^{2}r^{n}\\ \non
&\qquad \qquad\qquad- \psi(r)^{2}h(R^{N}(X(\omega),\nabla_{i}\wt\phi)\nabla_{i}\wt\phi, X(\omega))r^{n}\vol_{S^{n}}(\omega)dr \\ \non \\
&= \int_{0}^{\infty}\int_{S^{n}}\psi'(r)^{2}|X(\omega)|^{2}r^{n} + \psi(r)^{2}|\nabla^{\phi}_{j}X(\omega)|^{2}r^{n - 2}\\ \non
&\qquad \qquad\qquad- \psi(r)^{2}h(R^{N}(X(\omega),\nabla_{i}\phi)\nabla_{i}\phi, X(\omega))r^{n - 2}\vol_{S^{n}}(\omega)dr \\ \non \\
&= \int_{0}^{\infty}\psi'(r)^{2}r^{n}dr\int_{S^{2}}|X(\omega)|^{2}\vol_{S^{n}}(\omega) \\ \non
&\qquad\qquad\qquad+ \int_{0}^{\infty}\psi(r)^{2}r^{n - 2}dr\int_{S^{n}}|\nabla^{\phi}X|^{2} - h(R^{N}(X,\nabla_{i}\phi)\nabla_{i}\phi,X)\vol_{S^{n}} \geq 0
\end{align}
We have now isolated the $\omega$ and $r$ factors. The next natural step is to realize $\psi \not\equiv 0$ so that $\int_{0}^{\infty}\psi(r)^{2}r^{n - 2}dr > 0$. We then divide by this integral:
\begin{align}
\dfrac{\int_{0}^{\infty}\psi'(r)^{2}r^{n}dr}{\int_{0}^{\infty}\psi(r)^{2}r^{n - 2}dr}\int_{S^{n}}|X|^{2}\vol_{S^{n}} + \int_{S^{n}}|\nabla^{\phi}X|^{2} - h(R^{N}(X,\nabla_{i}\phi)\nabla_{i}\phi,X)\vol_{S^{n}} \geq 0
\end{align}
Taking the infimum over $\psi \in C_{c}^{\infty}((0,\infty);\BR)$ and applying \cite{SU84} Lemma 1.3 to compute this infimum, we achieve
\begin{align}
\int_{S^{n}}\dfrac{(n - 1)^{2}}{4}|X|^{2} + |\nabla^{\phi}X|^{2} - h(R^{N}(X,\nabla_{i}\phi)\nabla_{i}\phi,X)\vol_{S^{n}} \geq 0
\end{align}
which is the desired result.
\end{proof}
\bigskip
\section{Computing The Energy}
Let $G$ be a compact Lie group of type H, and let $\phi : S^{2} \to G$ be a harmonic map. In all of the calculations in this paper we will assume that $S^{2}$ carries the round metric of radius one and that the bi-invariant metric of $G$ has the Killing normalization. We now proceed to construct our variation and compute its energy. As mentioned before, by Theorem \ref{factorization} (\cite{BR90} Theorem 8.9), once we choose a Hermitian symmetric space $H$, we know that $\phi$ factors as a flag transformation of a lower energy harmonic map. Calling the lower energy harmonic map $\phi_{0}$ as in Theorem \ref{factorization}, we can write this formally as
\begin{align}
\phi = \phi_{0}\exp(\pi\sigma) \label{splitoffphi0}
\end{align}
for some flag factor $\sigma : S^{2} \to H \seq \mathfrak{g}$ for $\phi_{0}$. By \cite{BR90}, we can actually find a flag factor $\sigma$ satisfying the above factorization \eqref{splitoffphi0} for any Hermitian symmetric quotient $H$ of $G$. This will be crucial in the next section when we are estimating the norm of $\sigma.$ 

We also know that $E[\phi] - E[\phi_{0}] \geq 16\pi n_{G}$. Thus, it seems possible that $\sigma$ will be an energy decreasing variation of both $\phi$ and the cone map $\wt\phi.$ Note that $\sigma$ itself is not a vector field along the image of $\phi$ because it has image in the Lie algebra. We must multiply $\sigma$ on the left by $\phi$ to get a section of $\phi^{*}TG$. However, because the bundle $\phi^{*}TG = S^{2} \times \mathfrak{g}$ is trivial, it is very often convenient to view the variational field as just $\sigma.$ We will adopt this approach here. 

Now that we have our variational field, let us turn it into an actual variation of $\phi.$ This is simply $\phi_{t} = \phi_{0}\exp(t\sigma).$ At $t = 0$, we get $\phi_{0}$, and at $t = \pi$, we get $\phi.$

Our Lie group has a bi-invariant metric, so we may simplify the expression for the energy to the following expression for maps $\phi : M \to G$
\begin{align}
E[\phi] = \dfrac{1}{2}\int_{M}|\phi^{*}\theta|^{2}\vol_{M}
\end{align}
where $\theta$ is the left Maurer-Cartan form of $G.$ Therefore, the natural first step in computing the energy of $\phi_{t}$ is to compute $\phi_{t}^{*}\theta.$ As previously mentioned, it is more difficult to work directly with the real Lie algebra $\mathfrak{g}$, so we will extend $d\phi_{t}$ to get a map $d\phi_{t}^{\BC} : T^{\BC}S^{2} \to \mathfrak{g}^{\BC}$. This map will also be denoted $d\phi_{t}$ to keep our notation simple. We will use the Hermitian inner product induced by the Killing form of $\mathfrak{g}^{\BC}$ and Cartan involution of $\mathfrak{g}^{\BC}$. This takes the same values as minus the Killing form of $\mathfrak{g}$ on elements of $\mathfrak{g}$.

Recall also that for $x \in S^{2}$, $\ad(\sigma(x))$ is diagonalizable over $\BC$ and that the eigenspaces of $\ad(\sigma(x))$ in $\mathfrak{g}^{\BC}$ are orthogonal with respect to our choice of inner product. Also, because our flag factor is mapping into a Hermitian symmetric space, we know the only eigenvalues of $\ad(\sigma(x))$ are zero and $\pm\sqrt{-1}$ (see for example Chapter 4 of \cite{BR90}). If $v \in \underline{\mathfrak{g}}^{\BC}$, $v_{k}$ will denote the component of $v$ in $\mathfrak{g}_{k}$, the $\sqrt{-1}k$-eigenbundle of $\ad(\sigma).$ 

Below, we will reference the Maurer-Cartan form $\beta$ of $H$. For the precise definition of this form, see \cite{BR90} page 6. Briefly, given a homogeneous space $G/K$ with reductive factor $\mathfrak{m}$, we can identify $T(G/K)$ with a subbundle of $G/K \times \mathfrak{g}$ via the following composition of maps: 
\begin{align}
T(G/K) = G \times_{\Ad|_{K}}\mathfrak{g}/\mathfrak{k} = G \times_{\Ad|_{K}} \mathfrak{m} \hookrightarrow G/K \times \mathfrak{g}.
\end{align}
In other words, given an element $X_{x} \in T_{x}(G/K)$, we define the Lie algebra valued one form $\beta_{x}(X_{x}) \in \mathfrak{g}$ to be the $\mathfrak{g}$ component of the image of $X_{x}$ under the above composition of maps.

In the computation below, we will also use the following facts about flag factors derived in \cite{BR90} on page 92. Let $\beta$ be the Maurer-Cartan form of $H$. Then,
\begin{align}
d\sigma &= [\sigma^{*}\beta,\sigma] \label{dflag}\\
(\sigma^{*}\beta)^{1,0}_{1} = -\frac{1}{2}&(\phi_{0}^{*}\theta)^{1,0}_{1} \text{~and~} (\sigma^{*}\beta)^{0,1}_{-1} = -\frac{1}{2}(\phi_{0}^{*}\theta)^{0,1}_{-1} \label{flagtophi} \\
\sigma^{*}\beta \text{~has~} &\text{image in~} \mathfrak{g}_{-1} \oplus \mathfrak{g}_{1} \label{imflagbeta}
\end{align}

Let us now compute the energy $E[\phi_{t}]$ for any $t \in \BR$:
\begin{proposition}
Let $\phi_{t}: S^{2} \to G$ be as above. Then the energy is given by
\begin{align}
E[\phi_{t}] = E[\phi_{0}] + (1 - \cos(t))\int_{S^{2}}|\sigma^{*}\beta|^{2} - \dfrac{1}{2}|(\phi_{0}^{*}\theta)_{1}|^{2} - \dfrac{1}{2}|(\phi_{0}^{*}\theta)_{-1}|^{2}\vol_{S^{2}}
\end{align}
where $\beta$ is the Maurer-Cartan form of $H.$
\end{proposition}
\begin{proof} 
To begin, we need to compute $\phi_{t}^{*}\theta$:
\begin{align}
\phi_{t}^{*}\theta &= \phi_{t}^{-1}d\phi_{t} \\ \non \\
&= \exp(-t\sigma)\phi_{0}^{-1}d(\phi_{0}\exp(t\sigma)) \\ \non \\
&= \exp(-t\sigma)\phi_{0}^{-1}d(\phi_{0})\exp(t\sigma) + \exp(-t\sigma)d(\exp(t\sigma)) \\ \non \\
&= \Ad(\exp(-t\sigma))\phi_{0}^{*}\theta + \exp(-t\sigma)d(\exp(t\sigma)) \label{fullvariationpullback}
\end{align}
Let us compute the first term of \eqref{fullvariationpullback} using the definition of a flag factor: 
\begin{align}
\Ad(\exp(-t\sigma))\phi_{0}^{*}\theta &= \exp(-t\ad(\sigma))\phi_{0}^{*}\theta \\ \non \\
&= \sum_{k \geq 0}\dfrac{(-t)^{k}}{k!}\ad(\sigma)^{k}\phi_{0}^{*}\theta \\ \non \\
&= \sum_{k \geq 0}\dfrac{(-t)^{k}}{k!}\ad(\sigma)^{k}((\phi_{0}^{*}\theta)_{1} + (\phi_{0}^{*}\theta)_{0} + (\phi_{0}^{*}\theta)_{-1}) \\ \non \\
&= (\phi_{0}^{*}\theta)_{0} + \sum_{k \geq 0}\dfrac{(-t)^{k}}{k!}\left (\sqrt{-1}^{k}(\phi_{0}^{*}\theta)_{1} + (-\sqrt{-1})^{k}(\phi_{0}^{*}\theta)_{-1}\right ) \\ \non \\
&= (\phi_{0}^{*}\theta)_{0} + \sum_{k \geq 0}\dfrac{(-t\sqrt{-1})^{k}}{k!}(\phi^{*}_{0}\theta)_{1} + \sum_{k \geq 0}\dfrac{(t\sqrt{-1})^{k}}{k!}(\phi_{0}^{*}\theta)_{-1} \\ \non \\
&= (\phi_{0}^{*}\theta)_{0} + \exp(-t\sqrt{-1})(\phi_{0}^{*}\theta)_{1} + \exp(t\sqrt{-1})(\phi^{*}_{0}\theta)_{-1}
\end{align}
With this done, we will now compute the second term of \eqref{fullvariationpullback}. We will use a standard formula for the derivative of the exponential map. For an explanation of this fact, see Theorem 5.4 of \cite{Hal15}. 
\begin{align}
\exp(-t\sigma)d(\exp(t\sigma)) &= \exp(-t\sigma)d(\exp)_{t\sigma}d(t\sigma) \\ \non \\
&= \exp(-t\sigma)\exp(t\sigma)\sum_{k \geq 0}\dfrac{(-1)^{k}}{(k + 1)!}\ad(t\sigma)^{k}d(t\sigma) \label{dexp} \\ \non \\
&= \sum_{k \geq 0}\dfrac{(-1)^{k}t^{k + 1}}{(k + 1)!}\ad(\sigma)^{k}d\sigma
\end{align}
This can be simplified using \eqref{dflag}:
\begin{align}
\sum_{k \geq 0}\dfrac{(-1)^{k}t^{k + 1}}{(k + 1)!}\ad(\sigma)^{k}d\sigma &= \sum_{k \geq 0}\dfrac{(-1)^{k}t^{k + 1}}{(k + 1)!}\ad(\sigma)^{k}[\sigma^{*}\beta,\sigma] \label{differentialofsigma} \\ \non \\
&= \sum_{k \geq 0}\dfrac{(-t)^{k + 1}}{(k + 1)!}\ad(\sigma)^{k + 1}(\sigma^{*}\beta)
\end{align}
We then apply \eqref{imflagbeta}:
\begin{align}
\sum_{k \geq 0}\dfrac{(-t)^{k + 1}}{(k + 1)!}\ad(\sigma)^{k + 1}(\sigma^{*}\beta) &= \sum_{k \geq 0}\dfrac{(-t)^{k + 1}}{(k + 1)!}\ad(\sigma)^{k + 1}((\sigma^{*}\beta)_{1} + (\sigma^{*}\beta)_{-1}) \label{splittingbeta} \\ \non \\
&= \sum_{k \geq 0}\dfrac{(-t\sqrt{-1})^{k + 1}}{(k + 1)!}(\sigma^{*}\beta)_{1} + \sum_{k \geq 0}\dfrac{(t\sqrt{-1})^{k + 1}}{(k + 1)!}(\sigma^{*}\beta)_{-1} \\ \non \\
&= (\exp(-t\sqrt{-1}) - 1)(\sigma^{*}\beta)_{1} + (\exp(t\sqrt{-1}) - 1)(\sigma^{*}\beta)_{-1}
\end{align}
Combining the above lines, we observe:
\begin{align}
\exp(-t\sigma)d(\exp(t\sigma)) &= (\exp(-t\sqrt{-1}) - 1)(\sigma^{*}\beta)_{1} + (\exp(t\sqrt{-1}) - 1)(\sigma^{*}\beta)_{-1}
\end{align}
We will now compute the norm of \eqref{fullvariationpullback}:
\begin{align}
|\phi_{t}^{*}\theta|^{2} &= |\Ad(\exp(-t\sigma))\phi_{0}^{*}\theta + \exp(-t\sigma)d(\exp(t\sigma))|^{2} \\ \non \\
&= |(\phi_{0}^{*}\theta)_{0} + \exp(-t\sqrt{-1})(\phi_{0}^{*}\theta)_{1} + \exp(t\sqrt{-1})(\phi^{*}_{0}\theta)_{-1} \\ \non
&\qquad\qquad+ (\exp(-t\sqrt{-1}) - 1)(\sigma^{*}\beta)_{1} + (\exp(t\sqrt{-1}) - 1)(\sigma^{*}\beta)_{-1}|^{2} \\ \non \\
&= |(\phi_{0}^{*}\theta)_{0}|^{2} + |\exp(-t\sqrt{-1})(\phi_{0}^{*}\theta)_{1}|^{2} + |\exp(t\sqrt{-1})(\phi_{0}^{*}\theta)_{-1}|^{2} \\ \non
&\qquad\qquad+|(\exp(-t\sqrt{-1}) - 1)(\sigma^{*}\beta)_{1}|^{2} + |(\exp(t\sqrt{-1}) - 1)(\sigma^{*}\beta)_{-1}|^{2} \\ \non
&\qquad\qquad + (1 - \exp(-t\sqrt{-1}))\langle (\phi_{0}^{*}\theta)_{1}, (\sigma^{*}\beta)_{1} \rangle + (1 - \exp(t\sqrt{-1}))\langle (\sigma^{*}\beta)_{1}, (\phi_{0}^{*}\theta)_{1} \rangle \\ \non
&\qquad\qquad+(1 - \exp(t\sqrt{-1}))\langle (\phi_{0}^{*}\theta)_{-1}, (\sigma^{*}\beta)_{-1} \rangle + (1 - \exp(-t\sqrt{-1}))\langle (\sigma^{*}\beta)_{-1}, (\phi_{0}^{*}\theta)_{-1} \rangle
\end{align}
We then apply the definition of a flag factor:
\begin{align}
|\phi_{t}^{*}\theta|^{2} &= \label{splitcomponentsofphistartheta}|(\phi_{0}^{*}\theta)_{0}|^{2} + |(\phi_{0}^{*}\theta)_{1}|^{2} + |(\phi_{0}^{*}\theta)_{-1}|^{2} + (2 - 2\cos(t))(|(\sigma^{*}\beta)_{1}|^{2} + |(\sigma^{*}\beta)_{-1}|^{2}) \\ \non 
&\qquad\qquad+ (1 - \exp(-t\sqrt{-1}))\langle(\phi_{0}^{*}\theta)_{1}^{1,0}, (\sigma^{*}\beta)_{1}^{1,0} \rangle|dz|^{2} + (1 - \exp(t\sqrt{-1}))\langle (\sigma^{*}\beta)_{1}^{1,0}, (\phi_{0}^{*}\theta)_{1}^{1,0}\rangle|dz|^{2} \\ \non
&\qquad\qquad +(1 - \exp(t\sqrt{-1}))\langle (\phi_{0}^{*}\theta)_{-1}^{0,1},(\sigma^{*}\beta)^{0,1}_{-1}\rangle|d\bar{z}|^{2} + (1 - \exp(-t\sqrt{-1}))\langle (\sigma^{*}\beta)^{0,1}_{-1}, (\phi_{0}^{*}\theta)_{-1}^{0,1}\rangle|d\bar{z}|^{2} 
\end{align}
We continue on via an application of \eqref{flagtophi}:
\begin{align}
|\phi_{t}^{*}\theta|^{2} &= \label{subinthehalf}|\phi_{0}^{*}\theta|^{2} + (2 - 2\cos(t))|\sigma^{*}\beta|^{2} \\ \non
&\qquad + (1 - \exp(-t\sqrt{-1}))\left \langle(\phi_{0}^{*}\theta)_{1}^{1,0}, -\dfrac{1}{2}(\phi_{0}^{*}\theta)_{1}^{1,0}\right \rangle|dz|^{2} + (1 - \exp(t\sqrt{-1}))\left \langle -\dfrac{1}{2}(\phi_{0}^{*}\theta)_{1}^{1,0}, (\phi_{0}^{*}\theta)_{1}^{1,0}\right \rangle|dz|^{2} \\ \non
&\qquad +(1 - \exp(t\sqrt{-1}))\left \langle (\phi_{0}^{*}\theta)_{-1}^{0,1},-\dfrac{1}{2}(\phi_{0}^{*}\theta)^{0,1}_{-1}\right \rangle|d\bar{z}|^{2} + (1 - \exp(-t\sqrt{-1}))\left \langle -\dfrac{1}{2}(\phi_{0}^{*}\theta)^{0,1}_{-1}, (\phi_{0}^{*}\theta)_{-1}^{0,1}\right \rangle|d\bar{z}|^{2} \\ \non \\
&=  |\phi_{0}^{*}\theta|^{2} + (2 - 2\cos(t))|\sigma^{*}\beta|^{2} \\ \non
&\qquad\qquad -\dfrac{1}{2}|(\phi_{0}^{*}\theta)_{1}^{1,0}|^{2}|dz|^{2}((1 - \exp(-t\sqrt{-1})) + (1 - \exp(t\sqrt{-1}))) \\ \non
&\qquad\qquad -\dfrac{1}{2}|(\phi_{0}^{*})_{-1}^{0,1}|^{2}|d\bar{z}|^{2}((1 - \exp(t\sqrt{-1})) + (1 - \exp(-t\sqrt{-1}))) \\ \non \\
&= |\phi_{0}^{*}\theta|^{2} + (2 - 2\cos(t))|\sigma^{*}\beta|^{2} \\ \non
&\qquad\qquad -(1 - \cos(t))(|(\phi_{0}^{*}\theta)_{1}|^{2} + |(\phi_{0}^{*}\theta)_{-1}|^{2})
\end{align}
We now integrate our expression over $S^{2}$ and multiply by $\frac{1}{2}$ to conclude:
\begin{align}
E[\phi_{t}] = E[\phi_{0}] + (1 - \cos(t))\int_{S^{2}}|\sigma^{*}\beta|^{2} - \dfrac{1}{2}|(\phi_{0}^{*}\theta)_{1}|^{2} - \dfrac{1}{2}|(\phi_{0}^{*}\theta)_{-1}|^{2}\vol_{S^{2}}
\end{align}
which is the desired identity.
\end{proof}
Now, notice that when $t = \pi$, we get
\begin{align}
E[\phi] - E[\phi_{0}] = 2\int_{S^{2}}|\sigma^{*}\beta|^{2} - \dfrac{1}{2}|(\phi_{0}^{*}\theta)_{1}|^{2} - \dfrac{1}{2}|(\phi_{0}^{*}\theta)_{-1}|^{2}\vol_{S^{2}} = 2C
\end{align}
so our energy gain can be written down explicitly in terms of the flag transformation and $\phi_{0}$. Next, recall that the energy gain is actually bounded below by $16\pi n_{G}$, so we get that $C \geq 8\pi n_{G}.$

Now, going back to our expression for $E[\phi_{t}]$, we can finally estimate the second variation: 
\begin{align}
\nabla^{2}E[\phi](\sigma,\sigma) =\left . \dfrac{d^{2}}{dt^{2}}\right |_{t = \pi}E[\phi_{t}] = -C \leq -8\pi n_{G}. 
\end{align}
With all of this information in hand, we can now plug $\sigma$ into our cone stability inequality \eqref{conestabineq} to obtain:
\begin{align}
\dfrac{1}{4}\int_{S^{2}}|\sigma|^{2} \geq 8\pi n_{G}
\end{align}
We can simplify further by noticing that, for every $x$, $\sigma(x)$ lies in a single Adjoint orbit of $\mathfrak{g}$ by construction, so its norm is constant. Therefore, for a given $x \in S^{2}$,
\begin{align}
\dfrac{1}{4}\int_{S^{2}}|\sigma|^{2} = \dfrac{1}{4}(4\pi)|\sigma(x)|^{2} = \pi|\sigma(x)|^{2} \geq 8\pi n_{G}
\end{align}
It will be easier for us to compute $|\sigma(x)|$ if we can assume $\sigma(x)$ is of the form $\xi_{I}$, where $\xi_{I}$ is defined with respect to a maximal toral subalgebra $\mathfrak{t} \seq \mathfrak{g}$ and positive root system of our choice. To this end, recall that $\sigma$ takes values in the set of canonical elements of parabolic subalgebras and that conjugating a canonical element yields another canonical element. Therefore, $\sigma(x)$ is conjugate to some canonical element in $\mathfrak{t}$. Such elements can all be written in the form $\xi_{I}$ up to the action of the Weyl group. Therefore, $|\sigma(x)| = |\xi_{I}|$ and we can conclude that
\begin{align}
|\xi_{I}|^{2} \geq 8 n_{G}. \label{mainstabilitycondition}
\end{align}
Our next step is to choose $I$ so that we contradict \eqref{mainstabilitycondition} and hence also the cone stability of $\phi.$
\begin{remark}
Factoring $\phi$ completely, we see that the construction of \cite{BR90} provides an explicit homotopy between our map $\phi$ and a constant map. Furthermore, because the energy of $\phi_{t}$ is decreasing as $t$ goes from $\pi$ to zero, the energy along this homotopy is decreasing. 
\end{remark}
\begin{remark}
We could have also made the opposite substitution in \eqref{subinthehalf} to express $C$ in terms of only $\sigma.$
\end{remark}
\bigskip
\section{Estimating the Norms of the Canonical Elements}
The main objective here is choosing the correct canonical element to examine. We have some freedom over the choice of the canonical element because we are able to choose the Hermitian symmetric space that our flag factor $\sigma$ will map into before we factor $\phi$. The canonical elements corresponding to different Hermitian symmetric spaces are different and more importantly can have different norms, so this is a crucial step in our proof. The main fact we require is that the Hermitian symmetric spaces associated to a compact real Lie group are precisely the quotients of the complexified Lie group by a maximal parabolic subgroup corresponding to a simple root of coefficient one in the highest root (see for example \cite{BR90}, page 97). Similarly, we will also use the fact that a simple root with coefficient one in the highest root will have coefficient $\pm 1$ or zero in the rest of the roots. These facts are useful because they narrow our search quite a bit and make the computation easier. Indeed, in the case of $E_{6}$ and $E_{7}$, the existence of such a root alone is enough to guarantee the estimate we seek. 

The above tells us that we are searching for a canonical element of the form $\xi_{i}$. In other words, the index set of the corresponding parabolic subalgebra is $I = \{i\}$ for some $i$. Choosing the correct index $i$ is crucial because some Hermitian symmetric spaces correspond to parabolic subalgebras whose canonical element actually does satisfy \eqref{mainstabilitycondition}. An example of this situation is provided below in Example \ref{badcanonicalelementexample}. Once we have found the correct simple root, we can apply standard formulas to compute the norm of the corresponding canonical element. 

In order to actually do this, we need to look at the roots explicitly. This must be done on a case by case basis, and there are six types of Lie algebras to consider: the classical Lie algebras $A_{n}$, $B_{n}$, $C_{n}$, $D_{n}$, and two of the exceptional Lie algebras $E_{6}$, and $E_{7}.$ These are the only simple Lie algebras that allow for Hermitian symmetric quotients of the corresponding Lie groups.

In our computation, we will make use of the following formula: for $x$ and $y$ in the Cartan subalgebra used to define the roots, the Killing form can be written
\begin{align}
B(x,y) = \sum_{\alpha \in \Delta}\alpha(x)\alpha(y).
\end{align}
In particular, our canonical element lies in the Cartan subalgebra by definition, so we can compute its norm as
\begin{align}
|\xi|^{2} = -\sum_{\alpha \in \Delta}\alpha(\xi)^{2}. \label{normofcanformula}
\end{align}
See e.g. \cite{FH91} Section 14.2 for more information on the Killing form. To summarize, in this section we prove the following lemma:
\begin{lemma}\label{contradictionlemma}
Let $G$ be a compact simple Lie group of type $H$ that is not $\Sp(n)$ for $n \geq 8$. We can find a simple root $\alpha_{i}$, with coefficient one in the highest root, whose dual vector $\xi_{i}$ satisfies the strict inequality $8n_{G} > |\xi_{i}|^{2}$. In other words, \eqref{mainstabilitycondition} fails.
\end{lemma}
\begin{proof}
As mentioned, we should proceed by examining the specific cases. We follow \cite{FH91} for notation and describing the roots of these Lie algebras. See also \cite{BR90} chapter four and \cite{CP14} section five for more details on the explicit form of the canonical elements of $A_{n}.$ Note that \cite{CP14} is defining canonical elements in a slightly different context, so be careful when comparing definitions. For describing the roots of the remaining classical Lie algebras, see also \cite{CP15} section six. We set our problem up in a similar way.

We will first consider $A_{n}$. It turns out that every simple root has coefficient one in the highest root, so actually every maximal parabolic subgroup corresponds to a Hermitian symmetric space. Let us write down the roots more explicitly. Choose the maximal toral subalgebra $\mathfrak{t} \seq \mathfrak{su}(n + 1)$ consisting of the imaginary traceless diagonal matrices. We then obtain a Cartan subalgebra, $\mathfrak{t}^{\BC} \seq \mathfrak{sl}(n+1,\BC)$, given by the traceless diagonal matrices. We can compute that the roots of $\mathfrak{sl}(n + 1,\BC)$ with respect to $\mathfrak{t}^{\BC}$ are $L_{i} - L_{j}$ for $i \neq j$, where $L_{i} \in (\mathfrak{t}^{\BC})^{*}$ maps a diagonal element of $\mathfrak{sl}(n + 1,\BC)$ to the $i$th element of its diagonal. Next, a nice choice of positive roots is $L_{i} - L_{j}$ for $i < j$, and the corresponding simple roots are $\alpha_{i} = L_{i} - L_{i + 1}$ for $i = 1,\ldots,n.$ The highest root is then just the sum of all of the simple roots. Let us try $\alpha_{1}$. We see that $\alpha_{1}$ appears as a term in the roots $\pm(L_{1} - L_{j})$ for $j = 2,\ldots,n + 1$ and that the coefficient of $\alpha_{1}$ is always zero or $\pm 1$. Thus, because $\xi_{1}$ is defined to satisfy $\alpha_{i}(\xi_{1}) = \sqrt{-1}\delta_{i1}$,
\begin{align}
|\xi_{1}|^{2} = -\sum_{\alpha \in \Delta}\alpha(\xi_{1})^{2} = -2\sum_{j = 2}^{n + 1}((L_{1} - L_{j})(\xi_{1}))^{2} = -2\sum_{j = 2}^{n + 1}\sqrt{-1}^{2} = 2n
\end{align}
Now, quoting Table \ref{table:1}, $n_{\mathfrak{sl}(n + 1,\BC)} = n + 1$, so \eqref{mainstabilitycondition} reduces to $8(n + 1) \leq 2n$, i.e., $6n \leq -8$ which is never true for $n \geq 1$. Hence, \eqref{mainstabilitycondition} fails for $\SU(n)$.

For the remaining classical Lie algebras we define $\mathfrak{t}$ to be the real span of the $2n \times 2n$ matrices $E_{i} = E_{i,i} - E_{n + i,n + i}$, where $1 \leq i \leq n$ and $E_{i,i}$ has the entry $\sqrt{-1}$ at $(i,i)$ and zero everywhere else. The complexification of $\mathfrak{t}$, $\mathfrak{t}^{\BC}$, is a Cartan subalgebra for $B_{n}$, $C_{n}$, and $D_{n}$. The roots with respect to this Cartan subalgebra will now lie in $\sqrt{-1}\mathfrak{t}^{*}$, and we can write them down explicitly in terms of the elements $L_{i}$ of the basis for $\sqrt{-1}\mathfrak{t}^{*}$ dual to the $E_{i}$ in the sense that $L_{i}(E_{j}) = \sqrt{-1}\delta_{ij}$.  

Next, we consider $B_{n}$ for $n \geq 2$. We already know the result holds when $n = 1$ because $A_{1} = B_{1}.$ The roots of $B_{n}$ are $\pm L_{i} \pm L_{j}$ for $i \neq j$ and $\pm L_{i}$ for $1 \leq i \leq n$. Let us choose the positive roots to be $L_{i} \pm L_{j}$ for $i < j$ and $L_{i}$ for all $i.$ The simple roots are then $\alpha_{i} = L_{i} - L_{i + 1}$ and $\alpha_{n} = L_{n}.$ The highest root with respect to this choice of positive root system is $L_{1} + L_{2} = \alpha_{1} + \sum_{j = 2}^{n}2\alpha_{j}.$ Therefore, the only simple root leading to a Hermitian symmetric quotient is $\alpha_{1}.$ Thus, we choose $\xi_{1}$, the dual vector of $\alpha_{1}$, and compute its norm. Remember that all we need to do is find the roots in which $\alpha_{1}$ has coefficient $\pm 1$ so that we can apply \eqref{normofcanformula}. We see that $\alpha_{1}$ is a term in the roots $\pm L_{1} \pm L_{j}$ and $\pm L_{1}$, which is $4(n - 1) + 2 = 4n - 2$ different roots. Thus, we use \eqref{normofcanformula} to conclude that $|\xi_{1}|^{2} = 4n - 2$, as above. From Table \ref{table:1}, we know that $n_{\Spin(2n + 1)} = 2n - 1$. Plugging this into \eqref{mainstabilitycondition}, we achieve $4n - 2 \geq 8(2n - 1) = 16n - 8$. Equivalently, $n \leq \frac{1}{2}$, so \eqref{mainstabilitycondition} again fails.

The next case is $C_{n}$ for $1 \leq n \leq 7$. Here, the roots are $\pm L_{i} \pm L_{j}$ for $i \neq j$ and $\pm 2L_{j}$ for all $j.$ We see that a nice choice of positive roots is $L_{i} \pm L_{j}$ for $i < j$ and $2L_{j}$ for all $j$. The simple roots would then be $\alpha_{i} = L_{i} - L_{i + 1}$ and $\alpha_{n} = 2L_{n}.$ The highest root is $2L_{1} = \alpha_{n} + \sum_{j = 1}^{n - 1}2\alpha_{j}$. The only simple root with coefficient one in the highest root is $\alpha_{n}$, so this is the root we are searching for. We see that $\alpha_{n}$ has nonzero coefficient in $n(n + 1)$ of the roots and that Table \ref{table:1} reads $n_{\Sp(n)} = n + 1$, so \eqref{mainstabilitycondition} yields $n(n + 1) \geq 8(n + 1)$. Thus, \eqref{mainstabilitycondition} fails for $n < 8.$

The next case is $D_{n}$ for $n \geq 3.$ The roots are $\pm L_{i} \pm L_{j}$ for $i \neq j$, so we choose the positive root system $L_{i} \pm L_{j}$ for $i < j.$ The simple roots are then $\alpha_{i} = L_{i} - L_{i + 1}$ and $\alpha_{n} = L_{n - 1} + L_{n}.$ The highest root is $L_{1} + L_{2} = \alpha_{1} + \alpha_{n - 1} + \alpha_{n} + \sum_{j = 2}^{n - 2}2\alpha_{j}$. The simple root we will choose here is $\alpha_{1}$; if we chose $\alpha_{n - 1}$ or $\alpha_{n}$, the norm of the corresponding canonical element would be much bigger. Now, $\alpha_{1}$ has nonzero coefficient in $\pm L_{1} \pm L_{j}$, so $|\xi_{1}|^{2} = 4(n - 1)$. By Table \ref{table:1}, $n_{\Spin(2n)} = 2n - 2$. Plugging into \eqref{mainstabilitycondition}, we get $4n - 4 \geq 8(2n - 2) = 16n - 16$. In other words, $n \leq 1$, so \eqref{mainstabilitycondition} fails yet again. Note that we can ignore the case of $\Spin(4)$ because it is isomorphic to $\SU(2) \times \SU(2)$ and we already know the desired partial regularity result for $\SU(2).$

The last two cases are $E_{6}$ and $E_{7}$. We will handle these both at the same time because the same idea works for both of them. We know these have Hermitian symmetric quotients, so there exist simple roots $\alpha_{i}^{E_{6}}$ and $\alpha_{j}^{E_{7}}$ of the respective root systems corresponding to these quotients. For $E_{6}$, there are a total of 72 roots, so we see
\begin{align}
|\xi^{E_{6}}_{i}|^{2} = -\sum_{\alpha \in \Delta}\alpha(\xi_{i}^{E_{6}})^{2} \leq |\Delta| = 72.
\end{align}
On the other hand, reviewing Table \ref{table:1} and plugging into \eqref{mainstabilitycondition}, we should have $72 \geq 8(12) = 96$ which is false. Similarly, $E_{7}$ has 126 roots, so $|\xi_{j}^{E_{7}}|^{2} \leq 126.$ However, \eqref{mainstabilitycondition} says that $126 \geq 8(18) = 144$, which is again false.
\end{proof}
\begin{example}\label{badcanonicalelementexample}
We now provide an example to demonstrate that some canonical elements corresponding to a Hermitian symmetric space may satisfy \eqref{mainstabilitycondition}. Consider $A_{2n}$. One computes that $|\xi_{n}|^{2} = 2n(n + 1)$, while $8n_{\SU(2n + 1)} = 8(2n + 1).$ Thus, $\eqref{mainstabilitycondition}$ holds for $n \geq 8.$
\end{example}
\begin{remark}
Note that Table \ref{table:1} does not include values of $n_{\Spin(n)}$ for $n \leq 4$. It is well-known that Theorem \ref{mainthm} holds in these cases, so we do not need to worry about these values of $n$.
\end{remark}
\bigskip
\section{Lifting Harmonic Maps}
At this point, our work is nearly finished. In fact, we already have enough information to prove Theorem \ref{mainthm}. We will finish this proof in Section \ref{mainthmsection}.

Besides the proof of Theorem \ref{mainthm}, there are two more things we will show. In the current section, we will prove a lemma that will allow us to extend our result to certain quotients of the groups listed in Theorem \ref{mainthm}. Finally, in Section \ref{examplesection} we will construct examples of stable stationary harmonic maps $u : M \to G$ with codimension four singularities to show that our result is sharp.

To this end, we will now show that both weak harmonicity and \eqref{conestabineq} lift to the universal cover.
\begin{proposition} \label{liftingprop}
Let $(N,h)$ be a compact manifold with universal cover $p : (\wt N,p^{*}h) \to (N,h)$, and let $\phi : S^{2} \to N$ be a harmonic map satisfying our cone stability inequality \eqref{conestabineq}. There exists a lift $\wt\phi : S^{2} \to \wt{N}$ that is also a harmonic map satisfying \eqref{conestabineq}.
\end{proposition}
\begin{proof}
By general covering space theory, we may lift $\phi$ to the universal cover. Call the lift $\wt\phi.$ This entire proof boils down to the fact that $p$ is a local isometry, and in fact, the proof follows immediately from this. We will however provide all of the details in the computation.

Let $X \in \Gamma(S^{2},\phi^{*}TN).$ Then, we can lift $X$ pointwise to a vector field along $\wt\phi$ by the covering map $p : \wt N \to N.$ Call this lifted section $\wt X \in \Gamma(S^{2},\wt\phi^{*}T\wt N)$. We can do this using the isometric identification $p_{*} : T_{\wt\phi(x)}\wt N \to T_{\phi(x)}N$. $\wt X$ is smooth because $p$ is a local isometry. More precisely, in a small enough ball $B_{r}(x)$ around $x \in S^{2}$, by smoothness of $\wt\phi$, we can find an open set $U \supseteq \wt\phi(B_{r}(x))$ so that $p|_{U}$ is an isometry. Then, because $p|_{U}$ is an isometry, so is its inverse, i.e., $(p|_{U}^{-1})_{*}X|_{B_{r}(x)}$ is smooth. Similarly, if $\wt X \in \Gamma(S^{2},\wt\phi^{*}T\wt N)$, we can push it down to a section $X = p_{*}\wt X \in \Gamma(S^{2},\phi^{*}TN)$, so we get a complete identification of $\Gamma(S^{2}, \phi^{*}TN)$ with $\Gamma(S^{2},\wt\phi^{*}TN)$. This is exactly because we have already chosen the lift $\wt\phi$, so there is only one way to lift the variational field.

A quick way to see that $\wt\phi$ is harmonic is to realize that the above argument shows that variations of $\phi$ and $\wt\phi$ correspond, so let $\wt\phi_{t}$ be a variation of $\wt\phi$. We get a variation of $\phi$ denoted $\phi_{t} = p \circ \wt\phi_{t}.$ Then,
\begin{align}
E[\wt\phi_{t}] = \dfrac{1}{2}\int_{S^{2}}|\nabla\wt\phi_{t}|_{p^{*}h}^{2}\vol_{S^{2}} = \dfrac{1}{2}\int_{S^{2}}|\nabla\phi_{t}|_{h}^{2}\vol_{S^{2}} = E[\phi_{t}]
\end{align}
so $\frac{d}{dt}|_{t = 0}E[\wt\phi_{t}] = 0$ and $\wt\phi$ is harmonic.

For stability, we assume $\phi$ satisfies \eqref{conestabineq} and show $\wt\phi$ does too. Let us plug into our cone stability inequality \eqref{conestabineq}:
\begin{align}
\int_{S^{2}}\dfrac{1}{4}|\wt X|_{p^{*}h}^{2} + |\nabla^{\wt\phi}\wt X|_{p^{*}h}^{2} - (p^{*}h)(R^{p^{*}h}(\wt X, \nabla_{i}\wt\phi)\nabla_{i}\wt\phi, \wt X)\vol_{S^{2}}
\end{align} 
The first term is easy to compute: $|\wt X|_{p^{*}h} = |p_{*}\wt X|_{h} = |X|_{h}.$ For the second term, we need to be a little bit more careful. We compute this pointwise. Let $\{e_{1},e_{2}\}$ be a local orthonormal frame near $x \in S^{2}$. Then, at $x$,
\begin{align}
|\nabla^{\wt\phi}\wt X|^{2}_{p^{*}h} &= \sum_{i = 1}^{2}|\nabla_{e_{i}}^{\wt\phi}\wt X|^{2}_{p^{*}h} = \sum_{i = 1}^{2}\left |\nabla^{\wt\phi}_{e_{i}}\left (\wt X^{j}\dfrac{\p}{\p x^{j}}\right )\right |^{2}_{p^{*}h} \\ \non \\
&= \sum_{i = 1}^{2}\left |\left (\nabla_{e_{i}}\wt X^{j}\right )\dfrac{\p}{\p x^{j}} + \wt X^{j}\nabla^{\wt \phi}_{e_{i}}\dfrac{\p}{\p x^{j}}\right |^{2}_{p^{*}h}\\ \non \\
&= \sum_{i = 1}^{2}\left |\left (\nabla_{e_{i}}\wt X^{j}\right )\dfrac{\p}{\p x^{j}} + \wt X^{j}\nabla^{\wt N}_{\wt\phi_{*}(e_{i})}\dfrac{\p}{\p x^{j}}\right |^{2}_{p^{*}h} \\ \non \\
&= \sum_{i = 1}^{2}\left |p_{*}\left (\left (\nabla_{e_{i}}\wt X^{j}\right )\dfrac{\p}{\p x^{j}}\right ) + p_{*}\left (\wt X^{j}\nabla^{\wt N}_{\wt\phi_{*}(e_{i})}\dfrac{\p}{\p x^{j}}\right )\right |^{2}_{h}\\ \non \\
&= \sum_{i = 1}^{2}\left |\left (\nabla_{e_{i}}\wt X^{j}\right )p_{*}\left (\dfrac{\p}{\p x^{j}}\right ) + \wt X^{j}p_{*}\left (\nabla^{\wt N}_{\wt\phi_{*}(e_{i})}\dfrac{\p}{\p x^{j}}\right )\right |^{2}_{h}\\ \non \\
&= \sum_{i = 1}^{2}\left |\left (\nabla_{e_{i}}\wt X^{j}\right )p_{*}\left (\dfrac{\p}{\p x^{j}}\right ) + \wt X^{j}\nabla^{N}_{p_{*}\wt\phi_{*}(e_{i})}p_{*}\left (\dfrac{\p}{\p x^{j}}\right )\right |^{2}_{h}\\ \non \\
&= \sum_{i = 1}^{2}\left |\left (\nabla_{e_{i}}\wt X^{j}\right )p_{*}\left (\dfrac{\p}{\p x^{j}}\right ) + \wt X^{j}\nabla^{N}_{\phi_{*}(e_{i})}p_{*}\left (\dfrac{\p}{\p x^{j}}\right )\right |^{2}_{h}\\ \non \\
&= \sum_{i = 1}^{2}|\nabla^{\phi}_{e_{i}}p_{*}\wt X|^{2}_{h} = \sum_{i = 1}^{2}|\nabla^{\phi}_{e_{i}}X|^{2}_{h} = |\nabla^{\phi}X|^{2}_{h}
\end{align}
For the curvature terms, 
\begin{align}
(p^{*}h)(R^{p^{*}h}(\wt X, \nabla_{i}\wt\phi)\nabla_{i}\wt\phi, \wt X) &= h(R^{h}(p_{*}\wt X, p_{*}\wt\phi_{*}(e_{i}))p_{*}\wt\phi_{*}(e_{i}),p_{*}\wt X) \\
&= h(R^{h}(X,\nabla_{i}\phi)\nabla_{i}\phi,X)
\end{align}
Combining these three computations, we achieve:
\begin{align}
\int_{S^{2}}\dfrac{1}{4}|\wt X|_{p^{*}h}^{2} &+ |\nabla^{\wt\phi}\wt X|_{p^{*}h}^{2} - (p^{*}h)(R^{p^{*}h}(\wt X, \nabla_{i}\wt\phi)\nabla_{i}\wt\phi, \wt X)\vol_{S^{2}} \\
&= \int_{S^{2}}\dfrac{1}{4}|X|^{2}_{h} + |\nabla^{\phi}X|^{2}_{h} - h(R^{h}(X,\nabla_{i}\phi)\nabla_{i}\phi,X)\vol_{S^{2}} \geq 0.
\end{align}
Note that of course the same argument shows that if $\wt\phi$ satisfies \eqref{conestabineq}, then so does $\phi.$
\end{proof}
\bigskip
\section{Proof of the Main Theorem} \label{mainthmsection}
We will now recall the set up of our main theorem and prove it. Let $M$ be a compact manifold, and let $G$ be a compact simple Lie group that is not $\Sp(n)$ for $n \geq 8$, $E_{8}$, $F_{4}$, or $G_{2}$. We will use our work above to conclude a nonexistence result: there are no nonconstant cone stable maps $S^{2} \to G$. We then apply dimension reduction. From this, we conclude that stable stationary harmonic maps $u : M \to G$ are smooth away from a codimension four set. We will now reiterate this more formally.

\begin{proof}[Proof of Theorem \ref{mainthm}]
By Hsu's \cite{Hsu05} compactness theorem and a dimension reduction argument (see e.g. \cite{SU82} or \cite{HW99}), it suffices to show that there are no nonconstant cone stable maps $\phi : S^{2} \to G.$ We proceed by contradiction. If we did have such a cone stable $\phi$, then $\phi$ satisfies \eqref{mainstabilitycondition}. However, applying Lemma \ref{contradictionlemma} we find that this is impossible.
\end{proof}
Next, recall that compact simple simply connected Lie groups are the building blocks of all compact connected Lie groups. This is because the universal cover of a compact connected Lie group splits isometrically as $\BR^{k} \times G_{1} \times \ldots \times G_{\ell}$ where the $G_{j}$ are all compact simple simply connected Lie groups. 

Now, suppose $\phi : S^{2} \to K$ is a cone stable harmonic map into a compact Lie group $K$ whose universal cover $\wt K$ has isometric factors $G_{j}$ that are not $\Sp(n)$ for $n \geq 8$, $E_{8}$, $F_{4}$, or $G_{2}.$ Applying Proposition \ref{liftingprop}, we get a nonconstant harmonic map $\wt\phi : S^{2} \to \wt K$ that again satisfies \eqref{conestabineq}. We can then project $\wt \phi$ onto one of the factors of $\wt K$ such that the projection is not constant. Call this map $\wt\phi_{i}.$ Our map $\wt\phi_{i}$ is critical with respect to target variations along our chosen factor because $\wt\phi$ is critical with respect to variations of the form $(0,\ldots,0,X_{i},0,\ldots,0)$. This means $\wt\phi_{i}$ is harmonic. For essentially the same reason, we conclude that \eqref{conestabineq} also holds for $\wt\phi_{i}.$ Therefore, we get the same contradiction as in the proof of Theorem \ref{mainthm} and can conclude codimension four regularity for stable stationary harmonic maps into $K.$ Note that the same argument holds for many more isometric quotients of these groups $G.$ All we need is that the quotient map is a Riemannian covering map. 
\bigskip
\section{Examples of Singular Stable Stationary Harmonic Maps} \label{examplesection}
We will now explain why such a result is sharp. To this end we will prove the following theorem.
\begin{theorem} \label{stableconeimmersion}
Let $u : (S^{n},u^{*}g) \to (N,g)$ be a totally geodesic isometric immersion of a round sphere. Suppose that $S^{n}$ is a stable as a minimal submanifold of $N$ and that $n \geq 2$. Then the cone map associated to $u$, $\wt{u} : B_{1}^{n + 1} \to N$, is a stable stationary harmonic map.
\end{theorem}
Versions of this theorem have been proved before. Notably, Ohnita and Ugadawa \cite{OU90} proved that totally geodesic stable minimal isometric immersions of manifolds with stable harmonic identity map must be stable as harmonic maps. Further, Nagano and Sumi \cite{NS92} proved the analogous result for smooth $p$-harmonic maps. The reasoning in these cases is mostly the same as the reasoning in our case, so the above theorem is really nothing new. The only technical difficulty in our case is the singularity at the origin, but this does not turn out to be a problem. Nevertheless, we will proceed with a proof.
\begin{proof}
We will first show that the cone map is a stationary harmonic map. In dimension at least three it is well-known that the cone map will be stationary (see e.g. \cite{Nak03} Theorem 1.10). In dimension two, the cone map will be stationary if the underlying sphere map satisfies the balancing condition $\int_{S^{2}}|\nabla u|^{2}x^{i} = 0$ for all of the coordinate functions $x^{1},x^{2},x^{3}$ on $\BR^{3}$ restricted to $S^{2}$ (see e.g. \cite{Har97}). This is clearly satisfied in our case because $|\nabla u|$ is constant. This is because, for a local orthonormal frame $\{e_{1},e_{2}\}$ around a given point $x \in S^{2}$,
\begin{align}
|\nabla u|^{2} = |\nabla_{e_{1}}u|_{g}^{2} + |\nabla_{e_{2}}u|_{g}^{2} = |e_{1}|_{u^{*}g}^{2} + |e_{2}|_{u^{*}g}^{2} = 2.
\end{align}
Therefore the cone map will always be stationary. 

We now move on to stability. We will show this in three steps. The first step is to reduce the problem to showing stability in the tangential and normal directions separately. The next step then is to show that stability as a minimal submanifold implies the cone map is stable in the normal directions. The last step is to show stability in the tangent directions follows from stability of the projection map $B_{1}^{n + 1} \to S^{n}.$ Again, this all is very similar to \cite{OU90} and \cite{NS92} except we must now deal with the singularity at the origin.

The first step is to decompose the index form into tangent and normal components. To do this, let $X \in \Gamma(B_{1}^{n + 1}, \wt{u}^{*}TN)$ and decompose $X = X^{\top} + X^{\perp}$ into a component tangent to $S^{n}$ and a component normal to $S^{n}$. Then, by linearity and symmetry of the index form,
\begin{align}
I^{\wt{u}}(X,X) = I^{\wt{u}}(X^{\top} + X^{\perp},X^{\top} + X^{\perp}) = I^{\wt{u}}(X^{\top},X^{\top}) + I^{\wt{u}}(X^{\perp},X^{\perp}) + 2I^{\wt{u}}(X^{\top},X^{\perp}).
\end{align}
All we need to show is that the last term vanishes. Now let us recall that by definition we have
\begin{align}
I^{\wt{u}}(X^{\top},X^{\perp}) &= \int_{B_{1}^{n + 1}}\langle\nabla^{\wt{u}}X^{\top},\nabla^{\wt{u}}X^{\perp} \rangle - g(R^{N}(X^{\top},\nabla_{i}\wt{u})\nabla_{i}\wt{u},X^{\perp})
\end{align}
To show this is zero, we simply apply the totally geodesic condition. For totally geodesic submanifolds, the Levi-Civita connection preserves the tangent and normal decomposition. Therefore, $\langle \nabla^{\wt{u}}X^{\top},\nabla^{\wt{u}}X^{\perp}\rangle = 0$ away from zero.  Similarly, the curvature $R^{N}(X,Y)Z$, for $X, Y, Z \in TS^{n}$, will remain tangent, so the curvature terms also vanish away from zero. These two observations imply that if we integrate over any annulus not including zero that the integral will be zero, so $I^{\wt u}(X^{\top},X^{\perp}) = 0$. Thus, the index form splits as we claimed.

Next we show stability in the normal directions. Recall that a totally geodesic minimal submanifold $u : M \seq N$ is said to be stable if it is an index zero critical point of the area functional with respect to normal variations. In particular, for all compactly supported normal variations $X \in \Gamma(M, T^{\perp}M)$ we have
\begin{align}
\int_{M}|\nabla^{\perp}X|^{2} - g(R^{N}(X,e_{i})e_{i},X) \geq 0.
\end{align}
where $\{e_{i}\}$ is a local orthonormal frame near a given point in $M.$ This is precisely the stability inequality for harmonic maps for normal variations. Therefore, our map $u$ is stable with respect to normal variations. We just need to extend this to cone maps. To this end, let $X \in \Gamma(B_{1}, \wt{u}^{*}T^{\perp}M).$ We will additionally need the associated vector fields $X_{r} \in \Gamma(S^{n},u^{*}TN)$ defined by $X_{r}(\omega) = X(r\omega).$ These vector fields are useful because $r^{-1}|\nabla^{u}_{v}X_{r}| = |\nabla^{\wt{u}}_{v}X|$ when $v$ is a vector along the sphere. Also, as we have done before, when computing norms we choose the local orthonormal frame near a given $x \in B_{1}^{n + 1}$ to be $\{\frac{\p}{\p r},e_{1},\ldots,e_{n}\}$, where $e_{1},\ldots,e_{n}$ is a local orthonormal frame along the sphere. Plugging $X$ into the stability inequality we see:
\begin{align}
I^{\wt{u}}(X,X) &= \int_{B_{1}^{n + 1}}|\nabla^{\wt{u}}X|^{2} - g(R^{N}(X,\nabla_{i}\wt{u})\nabla_{i}\wt{u},X)dx \\
&= \lim_{\epsilon \to 0}\int_{\epsilon}^{1}\int_{S^{n}}|\nabla^{\wt{u}}X|^{2}r^{n} - g(R^{N}(X,\nabla_{i}\wt{u})\nabla_{i}\wt{u},X)r^{n}drd\omega \\ \non \\
&= \lim_{\epsilon \to 0}\int_{\epsilon}^{1}\int_{S^{n}}|\nabla^{\wt{u}}_{r}X|^{2}r^{n} + \sum_{j = 1}^{n}(|\nabla^{\wt{u}}_{j}X|^{2} - g(R^{N}(X,\nabla_{i}\wt{u})\nabla_{i}\wt{u},X))r^{n}drd\omega \\ \non \\
&\geq \lim_{\epsilon \to 0}\int_{\epsilon}^{1}\int_{S^{n}}\sum_{j = 1}^{n}(|\nabla^{\wt{u}}_{j}X|^{2} - g(R^{N}(X,\nabla_{i}\wt{u})\nabla_{i}\wt{u},X))r^{n}drd\omega \\ \non \\
&= \lim_{\epsilon \to 0}\int_{\epsilon}^{1}\int_{S^{n}}\sum_{j = 1}^{n}(|\nabla_{j}^{u}X_{r}|^{2} - g(R^{N}(X_{r},\nabla_{i}u)\nabla_{i}u,X_{r}))r^{n - 2}drd\omega \\ \non \\
&\geq 0.
\end{align}
The last inequality of course follows from the fact that $S^{n}$ is stable as a minimal submanifold. This concludes the proof of stability in the normal directions. 

Let us now move on to why $\wt{u}$ is stable in the tangent directions. Let $\wt f$ be the cone map associated to the identity map on $S^{n}.$ The point here is that we can write any tangential vector field along $\wt u$ as the pushforward under $u$ of a vector field along $\wt f$ because $\wt u = u \circ \wt f$. We will reduce the question of stability of $\wt u$ to the well-known fact that $\wt f$ is a stable harmonic map. 

Now, let $X \in \Gamma(B_{1}^{n + 1},\wt f^{*}TS^{n}).$ Then, $u_{*}(X) \in \Gamma(B_{1}^{n + 1},\wt u^{*}TN).$ Next, let $x \in B_{1}^{n + 1}$ be nonzero, and let $\{e_{j}\}$ be a local orthonormal frame near $\wt f(x)$ in $S^{n}$. Then, $\{u_{*}(e_{j})\}$ can be extended to a local orthonormal frame near $\wt u(x).$ In particular, $\{u_{*}(e_{j})\}$ can be used to express $u_{*}(X) = u_{*}(X^{k}(x)e_{k}(\wt f(x))) = X^{k}(x)u_{*}(e_{k}(\wt f(x))).$

We now claim that $u_{*}(\nabla^{\wt f}_{v}X) = \nabla^{\wt u}_{v}u_{*}(X)$ for any $v \in T_{x}B_{1}^{n + 1}$. We compute this directly
\begin{align}
u_{*}(\nabla^{\wt f}_{v}X) &= u_{*}(\nabla^{\wt f}_{v}(X^{k}e_{k}(\wt f(x)))) \\ \non \\
&= u_{*}((\nabla_{v}X^{k})e_{k}(\wt f(x)) + X^{k}\nabla^{\wt f}_{v}e_{k}(\wt f(x))) \\ \non \\
&= u_{*}((\nabla_{v}X^{k})e_{k}(\wt f(x)) + X^{k}\nabla^{S^{n}}_{\wt f_{*}(v)}e_{k}(\wt f(x))) \\ \non \\
&= (\nabla_{v}X^{k})u_{*}(e_{k}(\wt f(x))) + X^{k}u_{*}(\nabla^{S^{n}}_{\wt f_{*}(v)}e_{k}(\wt f(x))) \\ \non \\
&= (\nabla_{v}X^{k})u_{*}(e_{k}(\wt f(x))) + X^{k}\nabla_{\wt u_{*}(v)}^{N}u_{*}(e_{k}(\wt f(x))) \\ \non \\
&= (\nabla_{v}X^{k})u_{*}(e_{k}(\wt f(x))) + X^{k}\nabla^{\wt u}_{v}u_{*}(e_{k}(\wt f(x))) \\ \non \\
&= \nabla^{\wt u}_{v}u_{*}(X).
\end{align} 
From here we can compute that if $\{v_{j}\}$ is an orthonormal basis for $T_{x}B_{1}^{n + 1}$ that
\begin{align}
|\nabla^{\wt u}u_{*}(X)|^{2} = \sum_{j = 1}^{n + 1}|\nabla^{\wt u}_{v_{j}}u_{*}(X)|^{2} = \sum_{j = 1}^{n + 1}|u_{*}(\nabla^{\wt f}_{v_{j}}X)|^{2} = \sum_{j = 1}^{n + 1}|\nabla^{\wt f}_{v_{j}}X|^{2} = |\nabla^{\wt f}X|^{2}.
\end{align}
Next we have to investigate the curvature term of the index form. We compute from the Gauss equation:
\begin{align}
g(R^{N}(u_{*}(X),\nabla_{i}\wt u)\nabla_{i}\wt u,u_{*}(X)) = (u^{*}g)(R^{S^{n}}(X,\nabla_{i}\wt f)\nabla_{i}\wt f, X).
\end{align}
Combining the above, we can see that
\begin{align}
\int_{B_{1}^{n + 1}(0) \setminus B_{\epsilon}^{n + 1}(0)}|\nabla^{\wt u}&u_{*}(X)|^{2} - g(R^{N}(u_{*}(X),\nabla_{i}\wt u)\nabla_{i}\wt u, u_{*}(X)) \\
&= \int_{B_{1}^{n + 1}(0) \setminus B_{\epsilon}^{n + 1}(0)}|\nabla^{\wt f}X|^{2} - (u^{*}g)(R^{S^{n}}(X,\nabla_{i}\wt f)\nabla_{i}\wt f, X) \geq 0.
\end{align}
so taking the limit as $\epsilon \to 0$ we get stability in the tangential directions.
\end{proof}

\begin{remark}
The above theorem can be generalized to other manifolds with cone stable identity map.
\end{remark}

Now, to produce our singular stable stationary harmonic maps, we turn to the so-called Helgason spheres constructed by Helgason in \cite{Hel66}. These are totally geodesic isometrically immersed spheres sitting inside compact irreducible symmetric spaces that are not real projective spaces. Ohnita \cite{Ohn87} showed that these spheres are stable as minimal submanifolds.

Looking at Table 1 in \cite{Ohn87} shows that we have found a stable stationary harmonic map with codimension four singular set in the Lie groups $\SU(n)$ for $n \geq 2$, $\Spin(n)$ for $n \geq 5$, $\Sp(n)$ for $n \geq 3$, $E_{6}$, $E_{7}$, $E_{8}$, $F_{4}$, and $G_{2}.$ This construction can be extended to many other compact Lie groups by constructing a singular stable stationary harmonic map into the universal cover and then projecting it down. Therefore, codimension four is sharp. 
\bigskip

\bibliography{codim4}{}
\bibliographystyle{plain}

\end{document}